\documentclass[11pt]{amsart}
\usepackage{amsmath,amsthm,amsfonts,amssymb,amscd,amsbsy}
\usepackage{caption,enumerate}
\usepackage{dsfont,lscape}
\usepackage{stmaryrd}
\usepackage[all]{xy}
\usepackage{hyperref}
\usepackage{cleveref}
\usepackage{graphicx}

\hypersetup{
    pdftoolbar=true,
    pdfmenubar=true,
    pdffitwindow=false,
    pdfstartview={FitH},
    pdftitle={},
    pdfauthor={},
    pdfsubject={},
    pdfkeywords={}
    pdfnewwindow=true,
    colorlinks=true, 
    linkcolor=blue,
    citecolor=blue,
    urlcolor=black,
}

\allowdisplaybreaks

\newtheorem{theorem}{Theorem}[]
\newtheorem{lemma}[theorem]{Lemma}
\newtheorem{proposition}[theorem]{Proposition}

\theoremstyle{definition}

\theoremstyle{remark}
\newtheorem{remark}[theorem]{Remark}

\allowdisplaybreaks
\numberwithin{equation}{section}

\allowdisplaybreaks

\title
[Extremal domains for the $p$-Laplacian operator]
{Geometric properties of extremal domains for the $p$-Laplacian operator}

\author[Carvalho]{Francisco G. Carvalho}
 \address{\!\!\!\begin{tabular}{l}
 Universidade Federal do Piauí  \\
 Picos, Piauí, Brazil
 \end{tabular}
 }
 \email{franciscogsc.mat@ufpi.edu.br}

\author[Cavalcante]{Marcos P. Cavalcante}
 \address{\!\!\!\begin{tabular}{l}
 Universidade Federal de Alagoas \\
 Instituto de Matemática - IM \\
 Centro de Pesquisa em Matemática Computacional - CPMAT \\
 Maceió, Alagoas, Brazil 
 \end{tabular}
 }
 \email{marcos@pos.mat.ufal.br}
\allowdisplaybreaks
\numberwithin{equation}{section}
\numberwithin{theorem}{section}

\date{\today}

\begin{document}

\maketitle

\begin{abstract}
In this  paper, we explore the geometric properties of unbounded extremal domains for the $p$-Laplacian operator in both Euclidean and hyperbolic spaces.
Assuming that  the  nonlinearity grows at least as  the nonlinearity of the eigenvalue problem,  we prove that these domains exhibit remarkable geometric properties and cannot be arbitrarily wide.
    In two dimensions, we prove that such domains with connected complements  must necessarily be balls. 
In the hyperbolic space, we highlight the constraints on extremal domains and the geometry of their asymptotic boundaries.
\end{abstract}


\section{Introduction}

Let $\Omega \subset \mathbb{R}^n$ be a bounded domain in Euclidean space with $C^{2,\alpha}$ boundary, and let $\Delta_p u = \textrm{div}(\|\nabla u\|^{p-2} \nabla u)$ denote the $p$-Laplacian operator, where $1 < p < \infty$.
It is well known that the nonlinear eigenvalue problem
\begin{equation*}
	\left\{
	\begin{array}{rl}
		\Delta_p u+\lambda |u|^{p-2}u=0& \textrm{in } \Omega,\\
		u=0 & \textrm{on }  \partial\Omega,           
	\end{array}
	\right.
\end{equation*}
admits a principal eigenvalue associated to a positive eigenfunction.
This first eigenvalue, denoted by $\lambda_{1,p}(\Omega)>0$, 
is  simple, isolated, and can be variationally characterized as
\begin{equation}\label{lambda_variacoinal}
\lambda_ {1,p}(\Omega) = \inf \bigg\{\frac{\int_\Omega \|\nabla u\|^p \, dM}{\int_\Omega |u|^p\, dM}:
u\in W^{1,p}_0(\Omega), u\neq 0 \bigg \}.
\end{equation}

In particular, the first eigenvalue functional $\Omega \mapsto \lambda_ {1,p}(\Omega)$ is well-defined for all bounded smooth domains in $\mathbb{R}^n$, and we can compute its first variation (Hadamard formula) as done, for instance, in \cite{AnoopBobkovSasi, anisa15, garciamelian01, huang18}.

In fact, if $V$ is a smooth vector field along $\partial \Omega$, then the first derivative of $\lambda_{1,p}(\Omega)$ with respect to $V$ is given by
\begin{equation}\label{1stVariation}
    \delta_V[\lambda_{1,p}(\Omega)] = -(p-1)\displaystyle\int_{\partial\Omega}\langle V,\nu \rangle\left|\frac{\partial u_1}{\partial \nu}\right|^pd\sigma,
\end{equation}
where $u_1$ is the positive first eigenfunction associated with $\lambda_{1,p}(\Omega)$, whose $L^p$ norm is 1, and $\nu$ denotes the unit outward vector field along $\partial \Omega$.

In this setting, we say that $\Omega$ is  an
\emph{extremal domain for $\Delta_p$}
if it is a critical point for the first eigenvalue functional under volume-preserving variations. 
Note that the vector field $V$ generates a volume preserving variation 
if and only if
\[
\int_{\partial\Omega}\langle V,\nu \rangle d\sigma =0.
\]

Thus, from on (\ref{1stVariation}), we can conclude that $\Omega$ is an extremal domain for $\Delta_p$ if and only if $\frac{\partial u_1}{\partial \nu}$ is constant along the boundary. This condition is equivalent to the existence of a positive solution to the following overdetermined problem:
\begin{equation*}
\left\{
	\begin{array}{rl}
		\Delta_p u+\lambda u^{p-1}=0& \textrm{in } \Omega,\\
		u=0 & \textrm{on }  \partial\Omega,\\
		\frac {\partial u}{\partial \nu} =c & \textrm{on }  \partial\Omega.   
	\end{array}
	\right.
\end{equation*}

More generally, following the nomenclature in \cite{EspinarFarinaMazet} and \cite{EspinarMazet}, if $M^n$ is a Riemannian manifold, and $f:[0,\infty)\to \mathbb{R}$ is a Lipschitz  function, we define a smooth domain $\Omega\subset M^n$, whether bounded or not, as being \emph{$f$-extremal for the $p$-Laplacian operator} if there exists a positive solution $u\in C^{1,\alpha}(\Omega)$ to the problem:

\begin{equation}\label{problema1}
	\left\{
	\begin{array}{rl}
		\Delta_p u+f(u)=0& \textrm{in } \Omega,\\
		u=0 & \textrm{on }  \partial\Omega,\\
		\frac {\partial u}{\partial \nu} =c & \textrm{on }  \partial\Omega.
	\end{array}
	\right.
\end{equation}
Here, the $p$-Laplacian operator is defined as in the Euclidean space, but using the divergent and the gradient induced by the Riemannian metric of $M$.

Problem (\ref{problema1}) appears naturally in many problems 
from mathematical physics and is of great interest 
for mathematical analysis. 
In this article we will explore 
some geometrical properties of such domains in the Euclidean
space and in the hyperbolic space.

In the 50s, Aleksandrov \cite{aleksandrov58} introduced the celebrated reflection method to prove that an embedded closed hypersurface with constant mean curvature in the Euclidean space is a round sphere. Some years later, in 1971, Serrin \cite{Serrin71} adapted Aleksandrov's technique to solve some PDE problems. In particular, Serrin proved that if $p=2$ and $\Omega\subset \mathbb{R}^n$ is a bounded domain with a smooth boundary that admits a solution to problem (\ref{problema1}), then $\Omega$ is a ball. It is worth noting that Aleksandrov's theorem also holds for embedded hypersurfaces in the hyperbolic space or in the hemisphere and for more general elliptic equations of the principal curvature than the mean curvature (see \cite{Alexandrov62}), as well as Serrin's problem also holds for bounded domains in these same spaces, as proved in \cite{prajapatkumaresan98}. After Serrin, this technique became known as the moving planes method and has been successfully applied to a large number of problems, both in geometry and PDEs.

In light of Serrin's result, problem (\ref{problema1}) was investigated by Berestycki, Caffarelli, and Niremberg \cite{BCN} for unbounded domains $\Omega$ in Euclidean space when $p=2$. Their work led them to propose a conjecture that if $\mathbb{R}^n\setminus\Omega$ is connected, then $\Omega$ must be one of the following: a half-space, a ball, or the complement of a ball, a generalized cylinder, or the complement of a generalized cylinder. Counterexamples to this conjecture were first constructed by Sicbaldi \cite{sicbaldi10} when $n\geq 3$, and also by the works of Schlenk and Sicbaldi \cite{Schlenksicbaldi12}, Pacard, Del Pino, and Wei \cite{pacardpinowei15}, Fall, Minlend, and Weth \cite{Fall_Minlend_Weth17}, and Ros, Ruiz, and Sicbaldi \cite{RosRuizSicbaldi17}, \cite{rosruizsicbaldi20}. This latter article deals with the case of domains in the plane.

The investigation of problem (\ref{problema1}) when $p=2$ revealed  profound connections between the geometry of such domains and constant mean curvature hypersurfaces, as demonstrated in the work of Ros and Sicbaldi \cite{RosSicbaldi13}. In their paper, they employed  the moving planes method, assuming the  condition $f(t)\geq \lambda t$ and exploiting the linearity of the Laplacian operator.  On the other hand, the moving planes method can also be applied to the $p$-Laplacian operator, leading to a generalized version of Serrin's theorem. Specifically, if $\Omega\subset\mathbb{R}^n$ is a smooth bounded domain that admits a solution to problem (\ref{problema1}), then $\Omega$ must be a ball \cite{BrockHenrot, DamascelliPacella, GarofaloLewis}. Consequently, it is natural to extend the consideration of problem (\ref{problema1}) to unbounded domains, as Ros and Sicbaldi have done for the case of $p=2$ in \cite{RosSicbaldi13}.
 
In our first theorem, we prove that $f$-extremal domains for $\Delta_p$ cannot be arbitrarily wide. 
Namely, we denote by $R_\lambda>0$ the radius of a ball whose first eigenvalue of $\Delta_{p}$ is $\lambda$. Then we have:
 

\begin{theorem} \label{teo1}
Given $p>\frac{2n+2}{n+2}$, let $\Omega\subset 
\mathbb R^n$ be a domain admitting a positive weak solution $u\in C^{1,\alpha}(\Omega)$ of the equation
\[
\Delta_p u+f(u)=0,
\]
with  $f(t)\geq \lambda t^{p-1}$,  for some constant   $\lambda>0$. Then, $\Omega$ does not contain any closed ball of radius $R_\lambda$. 

Moreover, if $u$ also satisfies the boundary condition in the problem (\ref{problema1}), that is, if $\Omega$ is 
a $f$-extremal domain for the $p$-Laplacian, 
then, there exists $R_\lambda>0$ such that  either the closure $\overline{\Omega}$
does not contain any closed ball of radius $R_\lambda$ or
$\Omega$ is itself a ball of radius $R_\lambda$.
\end{theorem}

We point out that
$R_{\lambda}$ is the positive root of the function $z(t)$ that satisfies the following boundary value problem:
$$
\left\{ 
\begin{array}{rll}
(p-1)|z'(t)|^{p-2}z''(t)+\dfrac{(n-1)}{t}|z'(t)|^{p-2}z'(t)+\lambda|z(t)|^{p-2}z(t)=0, & & \\
z'(0)=0,                                                            & &  \\
z(0)=1.                                                             & &
\end{array}
\right.
$$

In the context of two dimensions, we can derive the following classification result:
\begin{theorem}\label{teo2}
Given $p>\frac{3}{2}$,
let $\Omega\subset 
\mathbb R^2$ be a $f$-extremal domain for the $p$-Laplacian, with  $f(t)\geq \lambda t^{p-1}$  
for some    $\lambda>0$. 
If  $\mathbb{R}^2\setminus\overline{\Omega}$ is connected, then  $\Omega$ is a ball.
\end{theorem}

This result follows from the following more general geometrical description of $f$-extremal domains of finite topology  (see Subsection \ref{planar} for more details).

\begin{theorem}\label{teo3}
Given $p>\frac{3}{2}$,
let $\Omega\subset 
\mathbb R^2$ be a finite topology 
$f$-extremal domain for the $p$-Laplacian, with  $f(t)\geq \lambda t^{p-1}$  
for some constant   $\lambda>0$. 
Then we have
   \begin{enumerate}
		\item  If $E$ is an end of $\Omega$, then $E$ stays at a bounded distance from a straight line;
		\item  $\Omega$ cannot have only one end;
		\item  if $\Omega$ has exactly  two ends, then there exists a line $L$ such that $\Omega$ is at a bounded distance from $L$.
	\end{enumerate}
\end{theorem}

Our proofs use the moving plane method and the strong comparison theorem, as established by Damascelli and Sciunzi in \cite{DamascelliSciunzi06}. The technical condition on the lower bound for $p$ is derived from their work.

Furthermore, the same problem can be investigated in more general Riemannian manifolds. Specifically, when considering hyperbolic space, we can employ the reflection method and derive results that extend the previous work by Espinar and Mao \cite{espinarmao18}. In their study, Espinar and Mao examined the geometry and topology of extremal domains in Hadamard manifolds, leading to the discovery of narrowness properties, an upper bounds estimate for the Hausdorff dimension of the asymptotic boundary, and a characterization of the asymptotic boundary of extremal domains.

In the hyperbolic context, we can demonstrate that several of the results obtained in \cite{espinarmao18} for the classical Laplacian still apply to the $p$-Laplacian operator. In this regard, our first result is presented in the following theorem.

\begin{theorem} \label{theoremH}
Given $p>\frac{2n+2}{n+2}$, let $\Omega\subset 
\mathbb H^n$ be a domain in the hyperbolic space admitting a positive weak solution $u\in C^{1,\alpha}(\Omega)$ of the equation
\[
\Delta_p u+f(u)=0,
\]
with  $f(t)\geq \lambda t^{p-1}$,  for some constant   $\lambda>\frac{(n-1)^p}{p^p}$. 
Then, $\Omega$ does not contain any closed ball of radius $R_\lambda$. 

Moreover, if $u$ also satisfies the boundary condition in the problem (\ref{problema1}), that is, if $\Omega$ is 
a $f$-extremal domain for the $p$-Laplacian in $\mathbb H^n$,  then, there exists $R_\lambda>0$ such that  either the closure $\overline{\Omega}$
does not contain any closed geodesic ball of radius $R_\lambda$ or
$\Omega$ is itself a geodesic ball of radius $R_\lambda$.
\end{theorem}

It is well known that while we can compactify the Euclidean space by adding a single point, the asymptotic boundary of the hyperbolic space $\mathbb{H}^n$ is diffeomorphic to a sphere $\mathbb{S}^{n-1}$. So, the behavior at infinity of noncompact sets in $\mathbb{H}^n$ can be very rich from the geometrical viewpoint. The next theorem, which is a consequence of Theorem \ref{theoremH}, asserts that the boundary at infinity of some $f$-extremal domains cannot be too large. We refer the reader to Section \ref{hyperbolic} for further definitions.

\begin{theorem}\label{boundary}
    Given $p>\frac{2n+2}{n+2}$, let $\Omega\subset \mathbb H^n$ be a $f$-extremal domain with  $f(t)\geq \lambda t^{p-1}$,  for some constant   $\lambda>\frac{(n-1)^p}{p^p}$. Then, there is no conical points in $\partial_\infty\Omega$ of radius $r>R_\lambda$. In particular, the Hausdorff dimension of $\partial_\infty\Omega$ is less than $n-1$.
\end{theorem}


We conclude this paper with a version of Theorem \ref{teo2} in hyperbolic space. However, we want to emphasize that the proof is significantly different and takes into account a classification result for extremal domains in $\mathbb{H}^2$ with only one point at the asymptotic boundary (Theorem \ref{horoball}), which is of independent interest.

\begin{theorem}\label{teo6}
Given $p>\frac{3}{2}$,
let $\Omega\subset 
\mathbb H^2$ be a $f$-extremal domain for the $p$-Laplacian, with  $f(t)\geq \lambda t^{p-1}$  
for some    $\lambda>1/p^p$. 
If  $\mathbb{H}^2\setminus\overline{\Omega}$ is connected, then  $\Omega$ is a geodesic ball.
\end{theorem} 

\subsection*{Acknowledgments}
This work is part of the doctoral thesis of the first author. The first author would like to express gratitude to the Federal University of Alagoas Mathematics Graduate Program for the support and hospitality during their Ph.D. studies. The second author was partially supported by CNPq (309733/2019-7, 201322/2020-0, and 405468/2021-0) and FAPEAL (E:60030.0000000323/2023)

\section{Preliminaries}
In this section, we present some important results that are essential for comprehending the forthcoming proofs.
\subsection{Strong comparison principle}
The first main tool is the following strong comparison theorem due to  Damascelli and Sciunzi,
adapted here to our proposals. This result is used in the proof of Theorem \ref{teo1} in the next section.

\begin{theorem}\cite[Theorem 1.4]{DamascelliSciunzi06}
\label{PCF}
Let $u,v\in C^1(\overline{B})$, where
$B$ is a bounded domain smooth domain in $\mathbb{R}^n$ and  
$\displaystyle\frac{2n+2}{n+2}<p<2$ or $p>2$. Suppose that either $u$ or $v$  is a weak solution of
\begin{equation*} 
\left\{
\begin{array}{rcccl}
	\Delta_p u+ f(u) &=& 0 & \textnormal{in} & B, \\[1mm]
	u &>& 0 & \textnormal{in} & B, \\[1mm]
 u &=& 0 & \textnormal{on} & \partial B .
\end{array}
\right.
\end{equation*}
where $f:[0,\infty)\rightarrow \mathbb{R}$ is
a positive  continuous function, Lipschitz in $(0,\infty)$. If	
 	\begin{equation}\label{PCFeq 2}
		-\Delta_p u-f(u)\leq -\Delta_p v-f(v)\ \ \textrm{and}\ \ u\leq v\ \ \text{in} \ \ B,
	\end{equation}
then $u\equiv v$ in $\Omega$, 
unless $u<v$ in $B$.
\end{theorem}

\begin{remark}  By a carefully reading of \cite{DamascelliSciunzi06} we can see that
Theorem \ref{PCF} still holds for bounded domains in 
the hyperbolic space.
\end{remark}

\subsection{A reflection principle for $f$-extremal domains} 
As a byproduct  of the proof of Theorem \ref{teo1}  in Section \ref{Proof_of_theorem1}
and also from previous works such as \cite{BrockHenrot,DamascelliPacella,GarofaloLewis}, it is known that the method of moving planes can be applied to equations involving the $p$-Laplacian operator. Specifically, in his thesis \cite[Teorema 3.4.2]{Gilberto}, the first author provided a proof of the following general symmetry result.
\begin{proposition} \label{reflection}
Given $p>1$, and a Lipschitz function 
$f:[0,\infty)\to \mathbb R$, let $\Omega\subset 
\mathbb R^n$  or $\Omega\subset 
\mathbb H^n$ be 
a $f$-extremal domain for the $p$-Laplacian
and let $P$ be a totally geodesic hyperplane that  intersects $\Omega$. 
If the halfspace $P^+$ determined 
by $P$ has a bounded connected component $C$ 
of $\Omega\setminus P$   and the reflection 
$C'$ of $C$ with respect to $P$ 
becomes internally tangent to the boundary of 
$\Omega$ at some point not on $P$ or $P$
reaches a position where it is orthogonal to the boundary of $\Omega$, then
$\Omega$ is symmetric with respect to $P$.
\end{proposition}

This proposition is used in the proofs of Theorems \ref{teo2}, \ref{teo3}, \ref{theoremH} and Proposition \ref{properties}.

\subsection{An inverse problem for the first eigenvalue of the $p$-Laplacian}
Using the variational charactherization 
(\ref{lambda_variacoinal}) 
we can easily see that  how $\lambda_{1,p}$ changes
under homotheties. 
In particular, if $B_R\subset \mathbb R^n$ is 
ball of radius $R$,
we have
$$\lambda_{1,p}(B_R)=\dfrac{\lambda_{1,p}(B_1)}{R^p}.$$
Thus, given  $\lambda>0$, there exists a unique  $R_{\lambda}>0$ such that
$\lambda_{1,p}(B_{R_{\lambda}})=\lambda$.
It means that  for each given $\lambda>0$ there exists 
$v\in C^{1,\alpha}(B_{R_\lambda})$ solution to
\begin{equation} \label{dom ext em Rn}
\left\{
\begin{array}{rcccl}
	\Delta_p v+\lambda |v|^{p-2}v &=& 0 & \textnormal{in} & B_{R_{\lambda}}, \\[1mm]
	v &>& 0 & \textnormal{in} & B_{R_{\lambda}}, \\[1mm]
 v &=& 0 & \textnormal{on} & \partial B_{R_{\lambda}}.
\end{array}
\right.
\end{equation}

The same result holds for geodesic balls  in the hyperbolic space  $\mathbb H^n$, for $\lambda$ 
bigger than the bottom of the spectrum of $\Delta_p$. Precisely,
\begin{proposition}\cite[Proposition 2.1]{MR4319403}
\label{Prop-Rlambda}
Given any $\lambda >\frac{(n-1)^p}{p^p}$
there exists unique $R_\lambda>0$
such that first eigenvalue of the  Dirichlet problem for the $p$-Laplacian on the
geodesic ball $B_{R_\lambda}\subset \mathbb H^n$
with radius $R_\lambda$  is precisely $\lambda.$
\end{proposition}
In particular, there exists a positive function
$v\in C^{1,\alpha}(B_{R_\lambda})$,
$B_{R_\lambda}\subset \mathbb H^n$, solution
to problem (\ref{dom ext em Rn}) in the hyperbolic space.


\section{Extremal domains in the Euclidean space}\label{Proof_of_theorem1}

\subsection{Narrowness properties of extremal domains}
In this subsection we present a proof of Theorem \ref{teo1} that overcomes the challenges arising from the non-linearity of the $p$-Laplacian operator.

\begin{proof}
As a first step, let us prove that 
there is no closed ball of radius $R_\lambda$ 
(or bigger)
contained in $\Omega$.  
Assume by contradiction that 
$\overline B_{R_\lambda}\subset \Omega$ and let 
$v\in C^{1,\alpha}(B_{R_\lambda})$ be the positive solution to 
\begin{equation}\label{v}
\left\{
\begin{array}{rll}
\Delta_p v+\lambda v^{p-1} &=& 0 \ \ \textrm{in} \ \  B_{R_{\lambda}}, \\
v &=& 0 \ \ \textrm{on} \ \ \partial B_{R_{\lambda}}
	\end{array}
	\right. 
	\end{equation}
normalized such that $\|v\|_{L^p(B_{R_\lambda})}=1$. 
 Given $\epsilon>0$, we note that
  $v_{\epsilon}=\epsilon v$ is also a positive
solution to (\ref{v}) and,
since $u> 0$ in 
$\overline B_{R_{\lambda}}$, we can take
$\epsilon>0$  such that 
$v_{\epsilon}(x)\leq u(x)$, for 
$x\in   \overline B_{R_{\lambda}}$
and $v_{\epsilon}(x_0)=u(x_0)$ for some
$ x_0 \in B_{R_{\lambda}}$.

Since $\Delta_p u+f(u)=0 $ in $\Omega$ and 
$f(u)\geq \lambda u^{p-1}$, we get
$$
-\Delta_p v_{\epsilon}-\lambda v_{\epsilon}^{p-1}\leq -\Delta_p u-\lambda u^{p-1},
$$
and from Theorem  \ref{PCF}, 
we have  $v_{\epsilon}\equiv u$
in $B_{R_{\lambda}}$, which is a contradiction.

\smallskip

Assuming that 
$\overline{B}_{R_\lambda}\subset\overline{\Omega}$, there must exist a point 
$y_0\in\partial\Omega\cap 
\partial\overline B_{R_{\lambda}}$. Let $v$ be the solution to problem (\ref{v}). As shown in \cite{bhathacharya88}, $v$ is radial, implying that $\frac{\partial v}{\partial \nu}$ is constant along $\partial B_{R_\lambda}$. Thus, we can find $\epsilon>0$ such that:

(1) Either $v_\epsilon(x)\leq u(x)$ in $B_{R_\lambda}$ and $v_\epsilon(x_0)=u(x_0)$ for some $x_0\in B_{R_\lambda}$.

\smallskip

(2) Or $v_\epsilon(x)<u(x)$ in $B_{R_\lambda}$ and
\begin{equation}\label{bordo}
\frac{\partial u}{\partial \nu}(y_0)=\frac{\partial v_\epsilon}{\delta \nu}(y_0)=c.
\end{equation}

The first case leads to a contradiction, as shown previously. Therefore, let us assume that assertion (2) holds and proceed to derive a contradiction using ideas similar to those in \cite{anisatoledo}.

Considering the equations satisfied by $u$ and $v_\epsilon$ within $B_{R_\lambda}$, along with the given condition on $f$, we obtain the following inequality:
\begin{equation*}
\begin{array}{rll}
-\textrm{div}\, \left(|\nabla u|^{p-2}\nabla u-|\nabla v_{\epsilon}|^{p-2}\nabla v_{\epsilon}\right) 
&=& -\Delta_p u+\Delta_p v_{\epsilon}\\
&=& f(u)-\lambda v_{\epsilon}^{p-1}\\
&\geq &\lambda u^{p-1}-\lambda v_{\epsilon}^{p-1}\\
&=& \lambda \left(u^{p-1}-v_{\epsilon}^{p-1}\right)\\
&\geq & 0.
\end{array}
\end{equation*}

That is,
\begin{equation}\label{DivA}
-\textrm{div}\, \left(A(\nabla u)-A(\nabla v_{\epsilon})\right)\geq 0, 
	\end{equation}
where $A(\xi)=|\xi|^{p-2}\xi$. 
Setting $A_i(\xi)=|\xi|^{p-2}\xi_i$, $i=1,\cdots,n$ and using the mean value theorem, we get
\begin{equation*}
\begin{array}{rll}
A(\nabla u)-A(\nabla v_{\epsilon}) &=& A(t\nabla u+(1-t)\nabla v_{\epsilon})\big |_{0}^{1} \\[1mm]
&=& \int_{0}^{1}\frac{d}{dt}A(t\nabla u+(1-t)\nabla v_{\epsilon})dt \\[1mm]
&=&\int_{0}^{1}\left(\alpha_1,\ldots,\alpha_n\right)dt \\[1mm]
&=&\big(\int_{0}^{1}\alpha_1dt,\ldots,\int_{0}^{1}\alpha_ndt\big),
\end{array}
\end{equation*}
where $w=u-v_{\epsilon}$ and
 $\alpha_i=\langle \nabla A_i,\nabla w\rangle=\displaystyle\sum_{j=1}^n\frac{\partial A_i}{\partial x_j}\frac{\partial w}{\partial x_j}$.
Thus,
$$
\textrm{div}\, \big(A(\nabla u)-A(\nabla v_{\delta})\big)=\sum_{i=1}^{n}\frac{\partial}{\partial x_i}\bigg\{\sum_{j=1}^{n}\bigg[\int_{0}^{1}\frac{\partial A_i}{\partial x_j}\big(t\nabla u(x)+(1-t)\nabla v_{\epsilon}(x)\big)dt\bigg]\frac{\partial w}{\partial x_j}\bigg\}.
$$

Noting that $A(\xi)=\nabla\Gamma (\xi)$, where	 $\Gamma(\xi)=\frac{|\xi|^p}{p}$, we can 
rephrase inequality (\ref{DivA}) as follows:
$$
-\sum_{i=1}^{n}\frac{\partial}{\partial x_i}\Big(\sum_{j=1}^{n}a_{ij}(x)\frac{\partial w}{\partial x_j}\Big)\geq 0,
$$
where	
$$
a_{ij}(x)=\int_{0}^{1}D_{ij}\Gamma\big(t\nabla u(x)+(1-t)\nabla v_{\epsilon}(x)\big)dt.
$$

We assert that there exists a neighborhood
$V\subset \overline B_{R_\lambda}$ of $y_0$, 
where the matrix $(a_{ij}(x))$
is uniformly positive definite. 

Let us define  
$g:[0,1]\times \overline{\Omega}\rightarrow\mathbb{R}$, 
as $g(t,x)=t\nabla u(x)+(1-t)\nabla v_{\epsilon}(x)$.
It follows that
\[
D_{ij}\Gamma(g(t,x))\geq K|g(t,x)|^{p-2},
\]
where $K=\min\{p-1,1\}$ (see \cite{anisatoledo}).

Since $v_\epsilon$ is a first eigenfunction
to the $p$-Laplacian operator in $B_{R_\lambda}$ it
is regular in a neighbourhood of the boundary.
More precisely (see \cite[Theorem 3.1]{barles88}), 
there are $\beta>0$, $\delta>0$ and $m>0$ such that
$v_\epsilon\in C^{2,\beta}(N^\delta (B_{R_\lambda}))$
and 
$|\nabla v_\epsilon|>m$ in $N^\delta (B_{R_\lambda})$,
where
$$N^\delta (B_{R_\lambda})=\{x\in B_{R_{\lambda}}
: dist(x,\partial B_{R_{\lambda}})<\delta\}.
$$

Note that $g$ in continuous in $[0,1]\times\Omega_\delta$ and
$|g(0,x)|=|\nabla v_{\epsilon}(x)|>m>0.$
In particular, there are $t_0>0$ and a neighbourhood $V$
of $y_0$ in $N^\delta (B_{R_\lambda})$ such that
$|g(t,x)|\geq c_0,$ for all $(t,x)\in [0,t_0]\times V$,
for some positive constant $c_0$.  It implies that 
\[
a_{ij}(x)\geq Kc_0^{p-2}t_0,
\]
for all $i,j\in \{1,\ldots,n\}$ and
for all $x\in V$, as we claimed.

So, we conclude that $w=u-v_\epsilon$ is a subsolution
to $Lw\leq 0$, where 
\begin{equation}\label{operador L}
L:=\displaystyle\sum_{i,j=1}^{n}\frac{\partial}{\partial x_i}\bigg(a_{ij}\frac{\partial}{\partial x_j}\bigg),
\end{equation}
is a uniformly elliptic linear operator. 

Since $w>0$ in $V$ and $w(y_0)=0$, applying the strong maximum principle yields
$$
\frac{\partial w}{\partial \nu}(y_0)<0
$$
which contradicts inequality (\ref{bordo}).
\end{proof}

As a consequence of this proof, we obtain a general symmetry result for $p$-extremal domains, as stated in Proposition \ref{reflection}. Once this symmetry result is established, we can follow the same steps as those taken by Ros and Sicbaldi in \cite[Section 4]{RosSicbaldi13} in order to derive important geometrical properties of $p$-extremal domains. These properties are summarized in the following proposition, which is fundamental for establishing a height estimate for planar domains in the next subsection.

\begin{proposition}\label{properties}
Let $\Omega\subset \mathbb R^n$ be an unbounded
$f$-extremal domain for the $p$-Laplacian
with $f(t)\geq \lambda t^{p-1}$,
and let $P$ be a totally geodesic hyperplane that does 
intersect $\Omega$. If the halfspace $P^+$ determined 
by $P$ has a bounded connected component $C$ 
of $\Omega\setminus P$, then 
\begin{enumerate}
    \item \label{graph} $\partial C\cap P^+$ is connected and it is
    graph over $\partial C\cap P$;
    \item \label{orthogonal} $\partial C\cap P^+$ is not orthogonal to
    $P$;
    \item If $C'$ denotes the reflection of $C$ with
    respect to $P$, then $\overline{C\cup C'}\subset
    \overline{\Omega}.$
\end{enumerate}
In particular, 
it is not possible to construct a half-ball of radius
$R_\lambda$ inside $C$ having its base on 
$\partial C\cap P$.
\end{proposition}

This last statement follows from our Theorem \ref{teo1}.

\subsection{Geometry of planar extremal domains}\label{planar} 

In the particular case of planar domains, we can
explore the low dimension setting to obtain stronger properties of $f$-extremal domains.

Along this subsection, 
$\Omega\subset\mathbb R^2$ 
will denote  an unbounded connect
planar domain. 
Given a straight line 
$L\subset \mathbb R^2$ 
we denote by $L^+$ and  $L^-$ 
the halfplanes determined by $L$ in  
$\mathbb{R}^2$.

\begin{lemma}\label{height} Assume that
$\Omega\subset\mathbb R^2$ is a
$f$-extremal domain with 
$f(t)\geq \lambda t^{p-1}$, for some
$3/2<p<\infty$.
Let $C$ be a bounded connected component of 
$\Omega\cap L^+$ and let  
$h(C)$ be the maximum distance  
between  $\partial C$ and $L$. 
Then  $h(C)\leq 3R_{\lambda}$.
\end{lemma}
\begin{proof} 
By employing an isometry, we can assume that 
\[L=\{(x,y)\in \mathbb{R}^2: y=0\}\]
and
\[L^+=\{(x,y)\in\mathbb{R}^2: y>0\}.\]

By utilizing assertion (\ref{graph}) stated in Proposition \ref{properties}, we can establish the existence of a smooth function 
\[g:[a,b]\subset L\rightarrow \mathbb{R}\]
satisfying $g(a) = 0 = g(b)$, and $g(x) > 0$ for all $x\in (a,b)$, such that the closure of $\partial C\cap L^+$ corresponds to the graph of $g$ over the interval $[a,b]$. In other words, we have
\[\overline{\partial C\cap L^+}=\{\left(x,g(x)\right): a\leq x\leq b)\}.\]

For the sake of simplicity and without loss of generality, we can assume that $0\in[a,b]$ and $g(0)=h(C)$ is the maximum height.

Assuming, by contradiction, that 
$h(C) > 3R_\lambda$, 
we can follow a similar argument as presented by Ros-Sicbaldi in \cite[Lemma 6.1]{RosSicbaldi13} to identify a new straight line denoted as $L_*$ that intersects $\partial C$ orthogonally and passes through the point $(a,0)$. Since we are working in a two-dimensional space, this straight line defines a subset $C^* = \Omega\cap L_*^+$ in such a way that $\partial C^*\cap L_*^+$ is orthogonal to $L_*$. However, this situation contradicts the assertion (\ref{orthogonal}) in Proposition \ref{properties}.
\end{proof}

Now, let us recall that a domain $\Omega$ in a complete metric space $M$ is classified as a \emph{finite topology domain} if any of the following conditions hold:

\begin{enumerate}
  \item The closure $\overline{\Omega}$ is compact.
  \item $\Omega$ is the complement of a compact set.
  \item There exists a compact set $K\subset M$ such that $\Omega\setminus K$ ia a finite union of disjoint sets, each of which is homeomorphic to a half-cylinder $B\times [0,\infty)$, where $B$ represents a metric ball in $M$.
\end{enumerate}

In the latter case, each individual component of $\Omega\setminus K$ is referred to as an \emph{end} of $\Omega$.

Using this terminology, we can now employ Propositions \ref{reflection} and \ref{properties}, along with Lemma \ref{height} to obtain a proof of Theorem \ref{teo3}, just following  the same steps as in \cite{RosSicbaldi13}.

Finally, let us assume that $\Omega\subset\mathbb{R}^2$ is a domain (thus connected) such that $\mathbb{R}^2\setminus \overline{\Omega}$ is connected. Then, we can deduce that $\Omega$ is diffeomorphic to a disc, and its boundary consists of a single connected curve. In particular, it is either a bounded domain or a domain with only one end. Assuming that $\Omega$ is an $f$-extremal domain, with $f(t)\geq \lambda t^{p-1}$ for some $p>3/2$, we can conclude from Theorem \ref{teo3} that $\Omega$ is bounded. In this case, we can apply Proposition \ref{reflection} to conclude that $\Omega$ is a disk. This establishes the proof of Theorem \ref{teo2}.
\section{Extremal domains in the hyperbolic space}\label{hyperbolic}

Let $\mathbb{H}^n$ be the hyperbolic space of curvature $-1$. 
In this section, we will identify $\mathbb{H}^n$ with the ball model $(\mathbb{B},g_{-1})$, equipped with the Poincaré metric $g_{-1}=\dfrac{4}{(1-|x|^2)^2}g_0$, where $g_0$ represents the Euclidean metric. 

As mentioned in the Introduction, given a smooth function $u$ defined on a domain $\Omega\subset\mathbb H ^n$ we can define $\Delta_p u =\text{div}(\|\nabla u\|^{p-2}\nabla u)$, just considering $\text{div}$, $\nabla$, and the norm with respect to the metric $g_{-1}$.

The first basic observation is that given any point $x\in \mathbb{H}^n$ and any direction $\eta\in T_x\mathbb{H}^n$, there exists a totally geodesic hyperplane $P$ passing through $x$ and orthogonal to $\eta$. Moreover, there is a 1-parameter family of totally geodesic hyperplanes containing $P$ that foliates $\mathbb{H}^n$, and the reflection with respect to $P$ is an isometry of $\mathbb{H}^n$. It is also well-known that the $p$-Laplacian operator remains invariant under isometries (see, for instance, \cite{Gilberto}). More precisely, we have the following lemma:

\begin{lemma}\label{lapla inv}
    Let $M$ be a Riemannian manifold, and $\Phi:M\rightarrow M$ be an isometry. If $\tilde{u}=u\circ \Phi^{-1}$ in $\Phi(\Omega)=\widetilde{\Omega}$, then:
    \begin{enumerate}
        \item $(\Delta_p u)\circ \Phi^{-1}=\Delta_p(u\circ\Phi^{-1})$, i.e., the $p$-Laplacian is invariant under isometries.
        \item If $u$ satisfies the equation $\Delta_p u+f(u)=0$ in $\Omega$, then $\tilde{u}$ satisfies the same equation in $\widetilde{\Omega}$, i.e., $\Delta_p \tilde{u}+f(\tilde{u})=0$ in $\tilde{\Omega}$.
    \end{enumerate}
\end{lemma}

All of this together implies that we can perform the moving planes method in hyperbolic space, as we have already seen in Proposition \ref{reflection} (see also \cite{prajapatkumaresan98}). In particular, taking into account Proposition \ref{Prop-Rlambda}, the proof of Theorem \ref{theoremH} follows in a manner similar to the proof of Theorem \ref{teo1}.

Our second fundamental observation pertains to the geometry of the asymptotic boundary of domains in hyperbolic space. Given two points $y$ and $z$ in the asymptotic boundary  $\partial_\infty \mathbb H^n = \mathbb S^{n-1}$, we denote by  $\gamma$ as the unique geodesic in $\mathbb H^n$ that connects $y$ and $z$. Specifically, $\gamma(+\infty) = y$ and $\gamma(-\infty) = z$. 

Now, for given numbers $r > 0$ and $s \in \mathbb R$, we introduce the concept of the \emph{cone at infinity based at $z$} as follows:
\[
\mathcal{C}_z(y,r,s) = \{x \in \mathbb H^n : d(x, \gamma(t)) \leq r, \text{ for all } t \geq s\}.
\]

Next, consider a domain $\Omega \subset \mathbb{H}^n$ such that $\partial_\infty \Omega \neq \emptyset$. Using the notation defined above, we define $z \in \partial_\infty \Omega$ as a \emph{conical point of radius $r$} if there exist $y \in \partial_\infty \mathbb{H}^n$ and $s \in \mathbb{R}$ such that $\mathcal{C}_z(y,r,s) \subset \Omega$. Notably, if $\Omega$ is a horoball with its boundary at infinity being $z$, then $z$ is considered a conical point for $\Omega$ for all radii $r > 0$.

By using this concept, the proof of Theorem \ref{boundary} can be established following the same ideas as in \cite[Theorem 2.7]{espinarmao18} and will be omitted.

To prove Theorem \ref{teo6}, we observe, as in the conclusion of Subsection \ref{planar}, that if $\mathbb{H}^2\setminus\overline{\Omega}$ is connected, then $\Omega$ is diffeomorphic to a disk, and its boundary, $\partial \Omega$, forms a single curve, since this is a topological nature property. 
We then conclude that either $\Omega$ is bounded (that is, $\partial_\infty\Omega=\emptyset$) and then a ball, or $\partial_\infty\Omega$ has only one point. In fact, other wise, we would have conical points of any radius at the asymptotic boundary of $\Omega$, which contradicts Theorem \ref{boundary}.
On the other hand, if $\partial_\infty\Omega$ has only one point, then the next theorem shows that $\Omega$ is a horoball, which also contradicts Theorem \ref{boundary}. 

This symmetric result for domains with only one point at infinity has its own interested and can be compare with Theorem 3.3 of \cite{espinarmao18} and Theorem 3.6 of \cite{EspinarFarinaMazet}, within the context of overdetermined problems, as well as to \cite{dCL} in the context of constant mean curvature hypersurfaces in hyperbolic space. 

\begin{theorem}\label{horoball}
    Let $\Omega\subset \mathbb H^2$ be a $f$-extremal domain for the $p$-Laplacian operator, and assume that $\partial_\infty\Omega$ consists of a single point. Then $\Omega$ is a horoball. 
\end{theorem}
\begin{proof}
Let $z_0=\partial_\infty\Omega$. Given any $z\in  \mathbb S^1=\partial_\infty \mathbb H^2$ let denote by $\gamma$ a geodesic of $\mathbb H^2$ joining $z=\gamma(-\infty)$ and $z_0=\gamma(+\infty)$.
    By applying an isometry of the hyperbolic plane, we can assume that $z_0=(1,0)\in \mathbb S^1$, $z=(-1,0)$ and $\gamma(s)=(\tanh (s/4),0)$, $s\in \mathbb R$. 
    Let denote by $\{\beta_t(s)\}_s$ the family of complete geodesics that intersects $\gamma^\perp(t)=(0,\tanh(t/4))$ orthogonally at  $\beta_t(0)$. Note that this family  foliates  $\mathbb H ^2 $ and $\beta_0(s)=\gamma (s)$.  Since $\partial_\infty\Omega$ consists of a single point, there exists $t_0$ such that $\beta_t$ does not intersects $\partial\Omega\subset \mathbb H^2$ for all $t<t_0$ and $\beta_{t_0}$ is tangent to $\partial \Omega$. 
    We can assume,  up to a reflection, that $t_0<0$.
    Now, we start the moving plane method in this situation. 
    For $t>t_0$ small, we denote by $C_t$ the connect component of $\Omega\cap \beta_t^-$, where $\beta_t^-$ is the half-space of $\mathbb H ^2$ determined by $\beta_t$ that contains $(0,-1)$ at its asymptotic boundary.
    We denote by $R_t:\mathbb H^2\to\mathbb H^2$ the reflection through $\beta_t$ and $C_t'=R_t(C_t)$. 
    Since $\partial\Omega$ is of class $C^2$,  if $t>t_0$ is small we have that $C_t$ is a pre-compact set and $C_t'\subset \Omega$.  

Note that $\beta_t$ cannot be a geodesic of symmetry to $\Omega$ for any $t<0$. In fact, other wise $\partial_\infty\Omega=\emptyset$. 
It follows that $C'_t\subset \Omega\cap \beta_t^+$ for all $t\in (t_0,0)$.

In the same way  as explained above, we can initiate the foliation towards $+\infty$ and decreasing the values of $t$ up to $t=0$. 
In this case, we necessarily  find a first value $t_1$ such that $\Omega$ intersects $\beta_t$ for $t>t_1$ small. 
We can then proceed in the same manner as described earlier to conclude that $C'_t$ is contained within $\Omega\cap \beta_t^-$ for all $t\in (0,t_1)$. This implies that $\Omega$ is symmetric with respect to $\gamma$.

\end{proof}

\bibliographystyle{amsplain}
\bibliography{references}

\end{document}